\documentclass{article}

\usepackage{typearea}
\typearea{8}
\usepackage{amsmath}

\usepackage{amssymb}
\usepackage{graphicx} 
\usepackage{wrapfig}
\usepackage{here}
\usepackage{amsthm}
\usepackage{cases}
\usepackage{url}

\newtheorem{dfn}{Definition}
\newtheorem{lem}{Lemma}
\newtheorem{thm}{Theorem}

\newtheorem{rmk}{Remark}
\newtheorem{prop}{Proposition}

\title{Discontinuous Riemann integrable functions emerging from cellular automata}
\author{Akane Kawaharada\footnote{E-mail: aka@kyokyo-u.ac.jp, Postal address: 1, Fujinomoricho, Fukakusa, Fushimi-ku, Kyoto-shi, Kyoto, 612-8522, Japan} \vspace{2mm}\\
Department of Mathematics, Kyoto University of Education}

\begin{document}

\maketitle

\begin{abstract}
This paper presents discontinuous Riemann integrable functions on the unit interval $[0, 1]$ derived from the dynamics of two-dimensional elementary cellular automata.
Based on the self-similarities of their orbits, we write down the numbers of nonzero states in the spatial and spatio-temporal patterns and obtain discontinuous Riemann integrable functions by normalizing the values.
We calculate the integrals of the two obtained functions over $[0, 1]$ and demonstrate the relationship between them.
\end{abstract}

\hspace{2.5mm} {\it Keywords} : cellular automaton, fractal, discontinuous Riemann integrable function

\section{Introduction}

A cellular automaton is a discrete dynamical system whose all valuables are discrete and a transition rule is given for each cell simultaneously. 
It is well known that cellular automata generate self-similar patterns.
For example, the one-dimensional elementary cellular automaton Rule $90$ generates the Sierpinski gasket and some two-dimensional elementary cellular automata generate isotopic crystal-like patterns.
Based on these self-similar structures, we constructed singular functions by normalizing the number of nonzero states in the spatio-temporal patterns of cellular automata.
A singular function is a function that is monotonically increasing (or decreasing) and continuous everywhere, with a zero derivative almost everywhere; for example, Salem's singular function \cite{salem1943, derham1957, ulam1934, yhk1997}.
We studied the relationship between Salem's singular function and elementary cellular automata, Rule $90$ and two two-dimensional elementary cellular automata \cite{kawanami2014, kawa2014a, kawanami2017, kawanami2019, kawanami2020}, and that between another new singular function and Rule $150$ \cite{kawa2021}. 

This paper presents new Riemann integrable functions with countable discontinuous points from two-dimensional elementary cellular automata. 
It is well known that there exist some Riemann integrable functions with discontinuous points. 
For example, Riemann's function is Riemann integrable and discontinuous at rational points if the denominator of an irreducible fraction is even \cite{riemann1867} and Thomae's function is also Riemann integrable and discontinuous at all rational points \cite{thomae1875}.
In this study, we consider two symmetrical two-dimensional elementary cellular automata, $T_a$ and $T_b$.
Figures~\ref{fig:st010} and \ref{fig:st696} show the spatio-temporal patterns of $T_a$ and $T_b$ from time step $0$ to $15$.
We calculate the number of nonzero states in their spatial and spatio-temporal patterns. 
By normalizing and limiting the dynamics of the numbers, we determine the sizes of self-similar sets and express the functions by an infinite sum of their sizes.
We show that the resulting functions are Riemann integrable and calculate the integrals of the functions. 

The remainder of this paper is organized as follows. 
Section \ref{sec:pre} describes the preliminaries concerning two-dimensional elementary cellular automata and the number of nonzero states in their spatial and spatio-temporal patterns. 
We review previous results about a two-dimensional cellular automaton, $T_{S0}$. 
In Section \ref{sec:main}, the number of nonzero states of the spatial and spatio-temporal patterns of $T_a$ and $T_b$ are given.
We define the given functions and express them using the self-similarities of the spatial patterns of the automata. 
We show that the resulting functions on $[0, 1]$ are bounded, uniformly continuous, and differentiable almost everywhere and calculate the definite integral of the functions over $[0, 1]$.
We also discuss the relationship between the two functions. 
Finally, Section \ref{sec:cr} discusses the findings of this study and highlights the possible future study directions.

\section{Preliminaries}
\label{sec:pre}

In this section, we present some definitions and notations for cellular automata. 
We also provide an overview of previous results about the number of nonzero states in the spatial patterns of a cellular automaton.

\subsection{Two-dimensional cellular automata}
 
In this study, we mainly focus on two-state two-dimensional cellular automata. 
Let $\{0,1\}$ be a binary state set and $\{0,1\}^{{\mathbb Z}^2}$ be the two-dimensional configuration space.
Suppose that $(\{0,1\}^{{\mathbb Z}^2}, T)$ is a discrete dynamical system consisting of the space $\{0,1\}^{{\mathbb Z}^2}$ and a transformation $T$ on $\{0,1\}^{{\mathbb Z}^2}$.
The $t$-th iteration of $T$ is denoted by $T^t$ and $T^0$ is the identity map on $\{0,1\}^{{\mathbb Z}^2}$.

\begin{dfn}
\begin{enumerate}
\item A two-dimensional elementary cellular automaton ($2$dECA) $(\{0,1\}^{{\mathbb Z}^2}, T)$ is given by 
\begin{align}
(T x)_{i, j} = 
f 
\begin{pmatrix}
x_{i, j+1} \\ x_{i-1, j} \quad x_{i, j} \quad x_{i+1, j} \\ x_{i, j-1}
\end{pmatrix}
= f
\begin{pmatrix}
U \\ L C R \\ D
\end{pmatrix}
\end{align}
for $(i, j) \in {\mathbb Z}^{2}$ and $x \in \{0,1\}^{{\mathbb Z}^{2}}$,  
where $f : \{0,1\}^{5} \to \{0,1\}$ depends on the five states of the von Neumann neighborhood; here, $f$ is a local rule. 
\item $2$dECA $(\{0,1\}^{{\mathbb Z}^{2}}, T)$ is a symmetrical $2$dECA (Sym-$2$dECA) 
if the local rule $f$ satisfies the following conditions:
\begin{align}
\label{eq:sym1}
&f \! \! 
\begin{pmatrix}
U \\  \! L C R \! \\ D
\end{pmatrix}
= f \! \! 
\begin{pmatrix}
D \\ \! L C R \! \\ U
\end{pmatrix}
= f \! \! 
\begin{pmatrix}
U \\ \! R C L \! \\ D
\end{pmatrix}
, \\
\label{eq:sym2}
&f \! \! 
\begin{pmatrix}
U \\ L C R \\ D
\end{pmatrix}
= f \! \! 
\begin{pmatrix}
L \\ \! D C U \! \\ R
\end{pmatrix}
= f \! \! 
\begin{pmatrix}
D \\ \! R C L \! \\ U
\end{pmatrix}
= f \! \! 
\begin{pmatrix}
R \\ \! U C D \! \\ L
\end{pmatrix}.
\end{align}
\end{enumerate}
\end{dfn}
The first set of equalities \eqref{eq:sym1} provides the left-right and top-bottom symmetries, while the second set \eqref{eq:sym2} provides the rotational symmetries. 
A local rule of Sym-$2$dECA is determined by a combination of twelve transitions (see Table~\ref{tab:u}). 

\begin{table}[H]
\caption{Local rules of Sym-$2$dECAs, $T_a$, $T_b$, and $T_{S0}$}
\label{tab:u}
\begin{center}
\scalebox{0.9}{
\begin{tabular}{c | c c c c c c c c c c c c} \hline \vspace{-1mm}
$U$ & $1$ & $0$ & $1$ & $0$ & $0$ & $0$ & $1$ & $0$ & $1$ & $0$ & $0$ & $0$\\ \vspace{-1mm}
$LCR$ & $111$ & $111$ & $010$ & $011$ & $010$ & $010$  &  $101$  &  $101$  &  $000$  &  $001$  &  $000$  & $000$\\
$D$ & $1$ & $1$ & $1$ & $1$ & $1$ & $0$ & $1$ & $1$ & $1$ & $1$ & $1$ & $0$\\ \hline
$(T_a x)_{i,j}$
& $\ast$ & $\ast$ & $\ast$ & $\ast$ & $\ast$ & $0$ & $1$ & $\ast$ & $1$ & $0$ & $1$ & $0$ \\
$(T_b x)_{i,j}$
& $1$ & $\ast$ & $1$ & $0$ & $1$ & $1$ & $1$ & $\ast$ & $\ast$ & $0$ & $1$ & $0$ \\ 
$(T_{S0} x)_{i,j}$
& $\ast$ & $\ast$ & $\ast$ & $\ast$ & $\ast$ & $0$ & $1$ & $\ast$ & $1$ & $0$ & $1$ & $0$ \\ 
\hline
\end{tabular}
}
\end{center}
\vspace{-2mm} \hspace{9cm} ($\ast$ is either $0$ or $1$.)
\end{table}

\begin{figure}[H]
\begin{center}
\includegraphics[width=1.\linewidth]{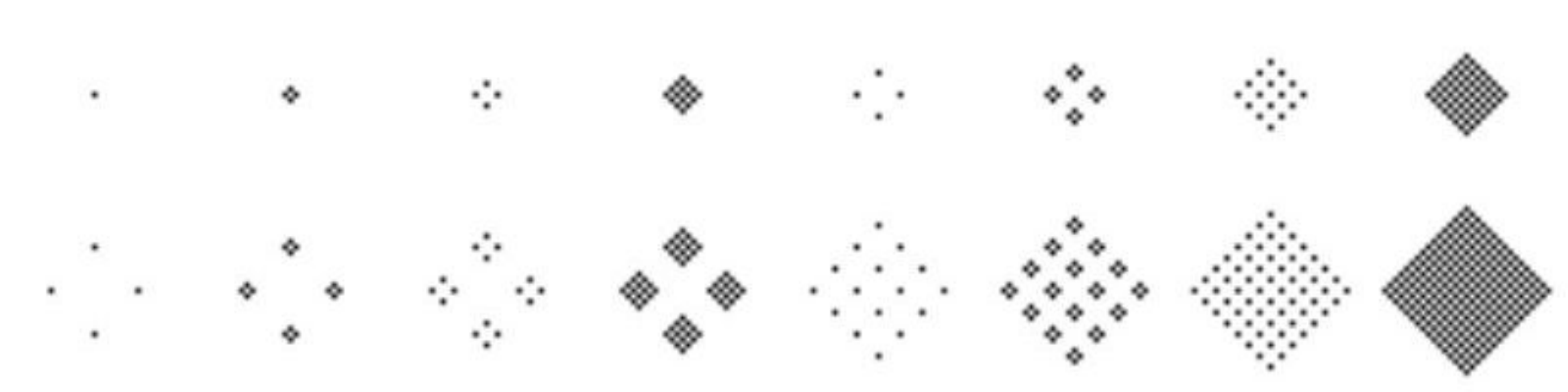}
\end{center}
\caption{Spatio-temporal pattern $T_{S0}^t x_o$ for $0 \leq t \leq 15$}
\label{fig:stS0}
\end{figure}

The configuration $x_{o} \in \{0,1\}^{{\mathbb Z}^2}$ is called the single site seed, wherein  
\begin{align}
({x_o})_{i,j} = \left\{
\begin{array}{ll}
1 & \mbox{if ${(i,j)} = (0, 0)$},\\
0 & \mbox{if ${(i,j)} \in {{\mathbb Z}^2} \backslash \{ (0, 0) \}$}.
\end{array}
\right. 
\label{eq:sss}
\end{align} 
In this study, we investigate the orbits from $x_o$ as the initial configuration. 
Figure~\ref{fig:stS0}, for example, shows the spatio-temporal pattern of Sym-$2$dECA $T_{S0}^t x_o$ from time step $0$ to $15$. 
The local rule of $T_{S0}$ is given in Table~\ref{tab:u}, where the symbol $\ast$ is either $0$ or $1$ because its value does not affect $T^t x_o$ for any $t$.

\begin{figure}[htbp]
\begin{minipage}{0.5\hsize}
\begin{center}
\includegraphics[width=70mm]{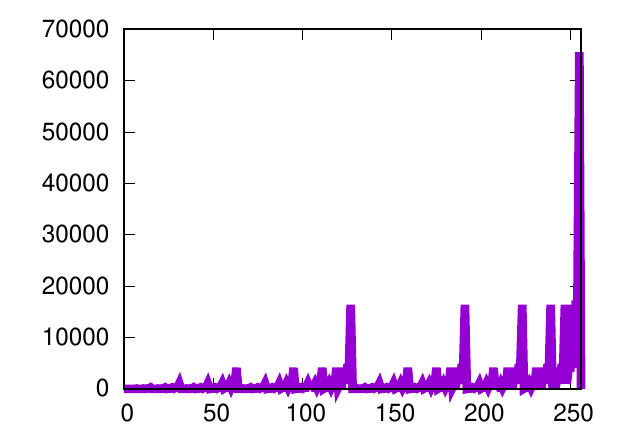}
\end{center}
\caption{$num_{S0}(t)$ for $1 \leq t \leq 256$}
\label{fig:num000}
\end{minipage}
\begin{minipage}{0.5\hsize}
\begin{center}
\includegraphics[width=70mm]{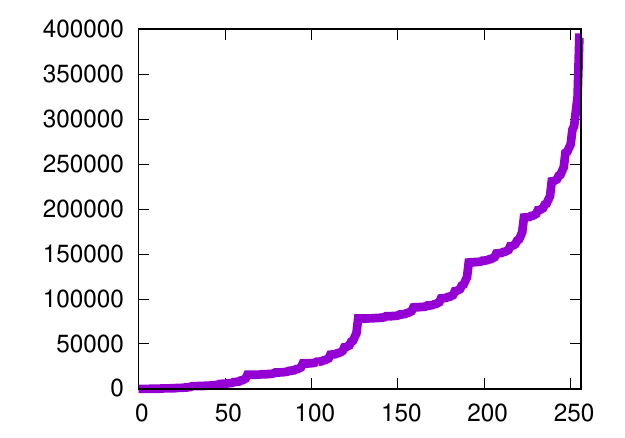}
\end{center}
\caption{$cum_{S0}(t)$ for $1 \leq t \leq 256$}
\label{fig:cum000}
\end{minipage}
\end{figure}

Next, we consider the number of nonzero states in the spatial and spatio-temporal patterns of Sym-$2$dECAs. 
For an automaton $T$, let $num_T(t)$ be the number of nonzero states in a spatial pattern $T^{t-1} x_o$ for $t>0$, and let $cum_T(t)$ be the cumulative sum of the number of nonzero states in a spatial pattern $T^{n-1} x_o$ from $n=1$ to $t(>0)$. 
We set $num_T(0) = cum_T(0) = 0$ owing to a technical reason.
Thus, we have
\begin{align}
num_T(t) = \sum_{(i, j) \in {{\mathbb Z}^2}} (T^{t-1} x_o)_{i, j}, \quad
cum_T(t) = \sum_{n = 0}^t \sum_{(i, j) \in {{\mathbb Z}^2}} (T^{n-1} x_o)_{i, j}.
\end{align}
We provide an overview of previous results regarding the number of nonzero states in the spatial or spatio-temporal pattern of Sym-$2$dECA $T_{S0}$.
In the case of $T_{S0}$, $num_T(t)$ and $cum_T(t)$ are denoted by $num_{S0}(t)$ and $cum_{S0}(t)$, respectively.
The graphs in Figures~\ref{fig:num000} and \ref{fig:cum000} show the dynamics of $num_{S0}(t)$ and $cum_{S0}(t)$ for $1 \leq t \leq 256$.
We obtained the following results about the number of nonzero states of $T_{S0}$.

\begin{prop}[\cite{kawanami2020}]
\label{prop:cumnumS0}
The values of $cum_{S0}$ and $num_{S0}$ are given by
\begin{align}
cum_{S0}(2^k) = 5^k \quad \mbox{and} \quad
num_{S0}(t+1) = 4^{\sum_{j=0}^{l-1} t_j}
\end{align}
for $t = \sum_{i=0}^{l-1} t_i 2^i$ and $k \in {\mathbb Z}_{\geq 0}$.
\end{prop}

\begin{prop}[\cite{kawanami2020}]
\label{prop:cumS0}
The normalized value of $cum_{S0}(t)/cum_{S0}(2^k)$ is represented by Salem's singular function $L_{\alpha}: [0,1] \to [0,1]$ with $\alpha = 1/5$, where
\begin{align}
L_{\alpha} (x):=
\left\{
\begin{array}{ll}
\alpha L_{\alpha}(2 x) & \ (0 \leq x < 1/2),\\
(1-\alpha) L_{\alpha}(2 x-1) + \alpha & \ (1/2 \leq x \leq 1).
\end{array}
\right. 
\label{eq:lb}
\end{align}
\end{prop}

Here, we consider the normalized value of $num_{S0}$ and define a function $H_k$ for $k \in {\mathbb Z}_{\geq 0}$.

\begin{rmk}
\label{rmk:hx}
We study $num_{S0}(t+1)/num_{S0}(2^k)$ for some $k \in {\mathbb Z}_{\geq 0}$.
By Proposition~\ref{prop:cumnumS0}, for $x = 1/2^k + \sum_{i=1}^k \hat{x}_i /2^i \in [1/2^k, 1]$, where $\hat{x}_i = t_{k-i}$, we have
\begin{align}
\label{eq:hk}
\frac{num_{S0}(t+1)}{num_{S0}(2^k)} = \frac{4^{\sum_{j=1}^k \hat{x}_j}}{4^k} =: H_k(x).
\end{align}
Because $H_k$ is only defined for at most $k$-digit binary decimals more than zero, $x \in [1/2^k, 1]$; we expand the domain to the unit interval $[0,1]$.
For $x=0$, we define $H_k(0)=0$ as $num_{S0}(0) = 0$ and 
for more than $k$-digit binary decimals, $x \in (0,1)$, we define $H_k(x) = 0$.
Thus, we have the function $H_k$ on $[0,1]$:
\begin{align}
H_k(x) = \left\{
\begin{array}{l l}
4^{\sum_{j=1}^k \hat{x}_j} /4^k & \mbox{if at most $k$-digit binary decimals except $x=0$},\\
0 & \mbox{if $x=0$ or more than $k$-digit binary decimals}.
\end{array}
\right .
\label{eq:Hconc}
\end{align}
The function $H_k$ is on $[0,1]$ and Riemann integrable because there are only finite discontinuous points on $[0,1]$. 

%
\end{rmk}

\section{Main results}
\label{sec:main}

Consider two orbits of Sym-$2$dECAs, $T_a$ and $T_b$, from the single site seed $x_o$. 
We calculate the numbers of nonzero states in the spatial and spatio-temporal patterns of $T_a$ and $T_b$ and show that the limits of the normalized numbers of $num_a(t)$ and $num_b(t)$ are represented by pathological functions with countable discontinuous points that are Riemann integrable.

\subsection{Numbers of nonzero states for Sym-$2$dECA $T_a$ and given function}

In this section, we examine the orbit of Sym-$2$dECA $T_a$.
Figure~\ref{fig:st010} shows the spatio-temporal pattern $T_a^t x_o$ for $0 \leq t \leq 15$.
According to Figure~\ref{fig:st010}, the spatial pattern for a finite time step $t$ consists of five parts. 
Four of them are congruent; they are also congruent with the spatial pattern of time step $t-2^k$; the remaining is a square with the length of $2^{k+1}-t-1$ cells at the center of each spatial pattern.
For the spatio-temporal pattern $T_a^{2^k} x_o$, the squares consist of a quadrangular pyramid whose apex is located at the bottom. 

\begin{figure}[H]
\begin{center}
\includegraphics[width=1.\linewidth]{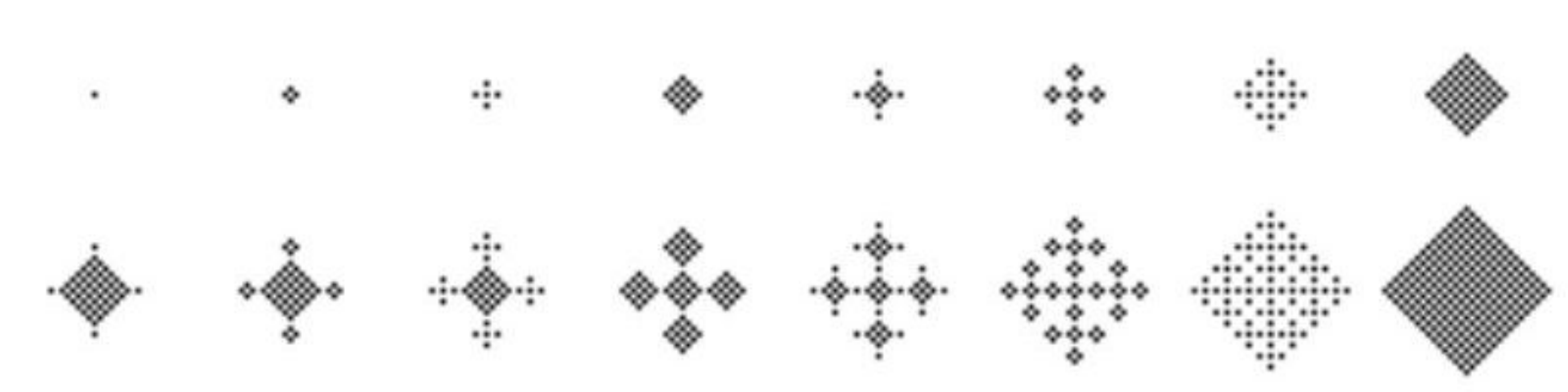}
\end{center}
\caption{Spatio-temporal pattern $T_a^t x_o$ for $0 \leq t \leq 15$}
\label{fig:st010}
\end{figure}

For Sym-$2$dECA $T_a$, we obtain the following results regarding $num_a(t)$ and $cum_a(2^k)$.

\begin{prop}
\label{prop:num10}
For $t = \sum_{i=0}^k t_i 2^i \geq 0$ with some $k \in {\mathbb Z}_{\geq 0}$ ($t_i=0$ for $i > k$), we have
\begin{align}
num_a(t) &= \sum_{i=0}^k t_{k-i} 4^{\sum_{j=0}^i t_{k-j+1}} \left( 2^{k-i+1} -t + \sum_{j=0}^i t_{k-j+1} 2^{k-j+1} \right)^2,\\
cum_a(2^k) &= \frac{1}{18} (2 \cdot 8^k + 8 \cdot 5^k + 9 \cdot 4^k - 2^k). 
\end{align}
\end{prop}

\begin{proof}
The value of $num_a(t)$ is calculated using the following recurrence equations. 
For $t=0$ and $1$, $num_a(0)=0$ and $num_a(1)=1$ are easily obtained.
Using the self-similarities of the spatial patterns, we obtain the third equation.
\begin{align}
\label{eq:4tk}
\begin{cases}
num_a(0) &= 0, \\
num_a(1) &= 1,\\
num_a(t) &= 4^{t_k} \, num_a(t-t_k 2^k) + t_k (2^{k+1}-t)^2.
\end{cases}
\end{align}
According to Equation~\eqref{eq:4tk}, we have $num_a(t-t_k 2^k) = 4^{t_{k-1}} \, num_a(t- t_k 2^k - t_{k-1} 2^{k-1}) + t_{k-1} (2^k-t + t_k 2^k)^2$. Hence,
\begin{align}
num_a(t) 
&= 4^{t_k} (4^{t_{k-1}} \, num_a(t- t_k 2^k - t_{k-1} 2^{k-1}) + t_{k-1} (2^k-t + t_k 2^k)^2) + t_k (2^{k+1}-t)^2\\
&= 4^{t_k + t_{k-1}} \, num_a(t- t_k 2^k - t_{k-1} 2^{k-1}) + 4^{t_k} t_{k-1} (2^k-t + t_k 2^k)^2 + t_k (2^{k+1}-t)^2.
\end{align}
Because $num_a \left(t- \sum_{i=0}^k t_{k-i} 2^{k-i} \right) = num_a(0) = 0$, we inductively calculate
\begin{align}
\label{eq:num01}
num_a(t) &= \sum_{i=0}^k t_{k-i} 4^{\sum_{j=0}^i t_{k-j+1}} \left( 2^{k-i+1} -t + \sum_{j=0}^i t_{k-j+1} 2^{k-j+1} \right)^2.
\end{align}

\begin{figure}[htbp]
\begin{minipage}{0.5\hsize}
\begin{center}
\includegraphics[width=50mm]{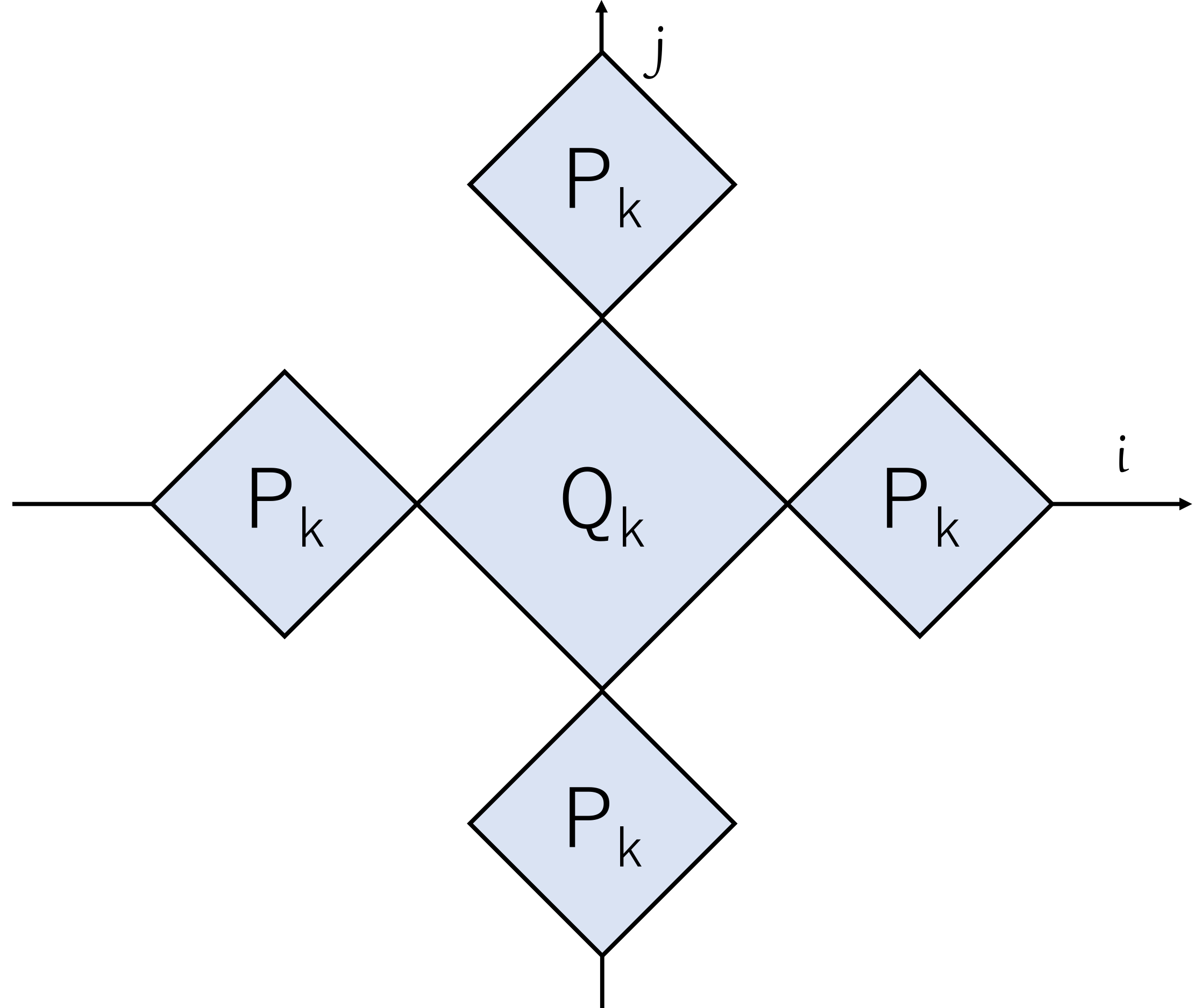}\\
$(a)$ Slice of quadrangular pyramids, four $P_k$s and one $Q_k$, on $T_a^t x_o$ and $T_b^t x_o$ for some $t$
\end{center}
\end{minipage}
\begin{minipage}{0.5\hsize}
\begin{center}
\includegraphics[width=50mm]{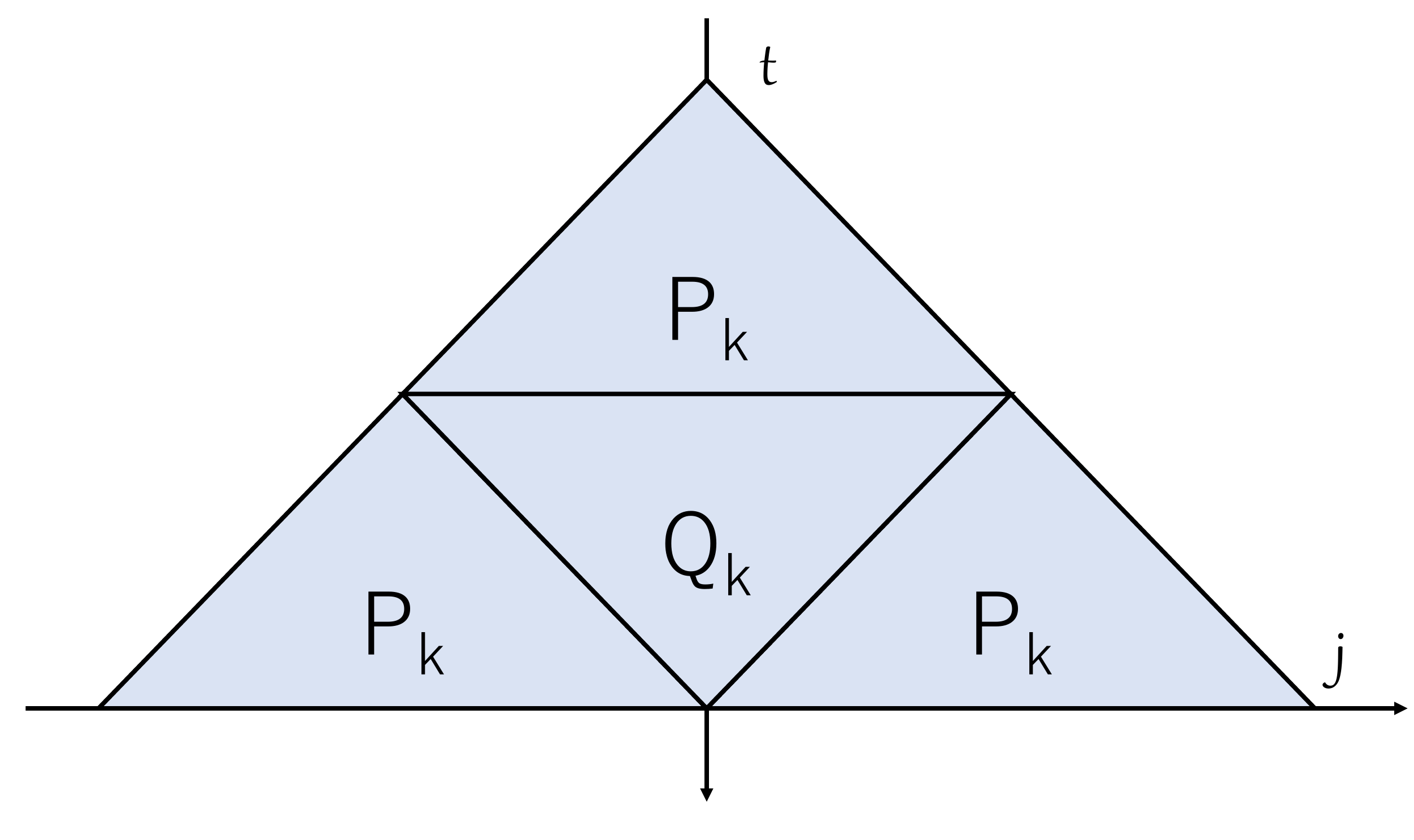}\\
$(b)$ Quadrangular pyramids, three $P_k$s and one $Q_k$, on the cutting planes of $\{T_a^t x_o\}_{t=0}^{2^k}$ and $\{T_b^t x_o\}_{t=0}^{2^k}$ when $i=0$
\end{center}
\end{minipage}
\caption{Self-similar quadrangular pyramids, $P_k$s and $Q_k$, on $\{T_a^t x_o \}_{t=0}^{2^k}$ and $\{T_a^t x_o \}_{t=0}^{2^k}$ for $k \in {\mathbb Z}_{\geq 0}$}
\label{fig:sph}
\end{figure}

Next, we calculate $cum_a(2^k)$. 
The spatio-temporal pattern $\{T_a^t x_o\}_{t=0}^{2^k}$ consists of six quadrangular pyramids, five $P_k$s, and one $Q_k$ (see Figure~\ref{fig:sph}). 
Let $u_k = cum_a(2^k)$ and let $v_k$ be the number of nonzero states in $Q_k$. Thus, 
\begin{align}
\begin{cases}
u_0 = 1,\\
v_{k+1} = \sum_{i=1}^{2^k-1} i^2 
= \frac{1}{6} 2^k (2^k-1) (2^{k+1}-1) 
= \frac{1}{6} (2 \cdot 8^k -3 \cdot 4^k + 2^k),\\
u_{k+1} = 5 u_k + v_{k+1}.
\end{cases}
\end{align}
Consequently, we have
\begin{align}
u_k &= \frac{4 \cdot 5^k}{9} + \frac{1}{18} (2 \cdot 8^k + 9 \cdot 4^k - 2^k)\\
&= \frac{1}{18} (2 \cdot 8^k + 8 \cdot 5^k + 9 \cdot 4^k - 2^k) = cum_a(2^k).
\end{align}
\end{proof}

\begin{figure}[htbp]
\begin{minipage}{0.5\hsize}
\begin{center}
\includegraphics[width=70mm]{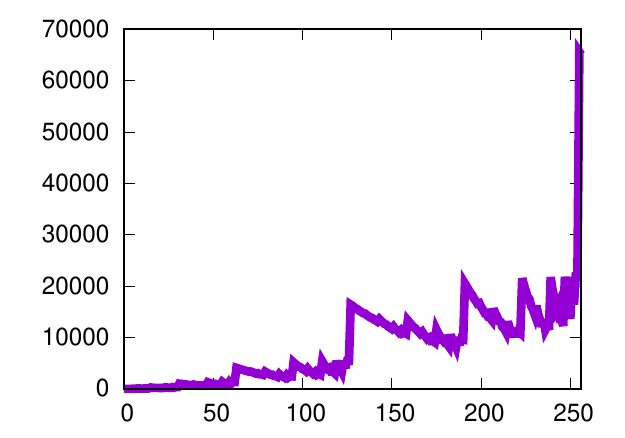}
\end{center}
\caption{$num_a(t)$ for $1 \leq t \leq 256$}
\label{fig:num010}
\end{minipage}
\begin{minipage}{0.5\hsize}
\begin{center}
\includegraphics[width=70mm]{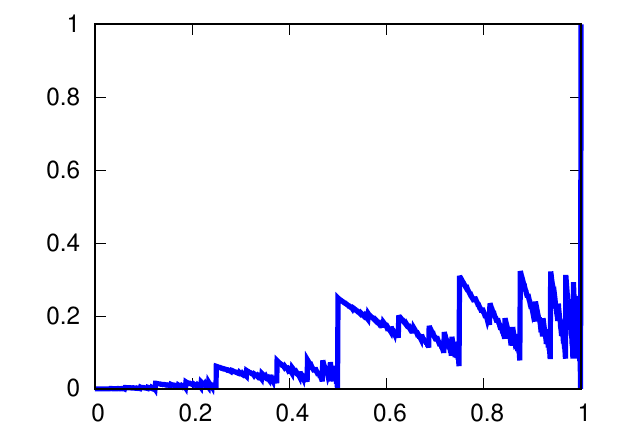}
\end{center}
\caption{$F(x)$}
\label{fig:Fx}
\end{minipage}
\end{figure}

The dynamics of $num_a(t)$ is plotted in Figure~\ref{fig:num010}.
We present a function $F$ by normalizing $num_a(t)$. 

\begin{dfn}
A nonnegative integer $t|_k (\leq 2^k)$ is expressed as $t|_k= \sum_{i=0}^k t_i 2^i$.
For $k > 0$ and $x = t|_k/2^k \in [0,1]$, 
we present $f_k(x) := num_a(t|_k)/num_a(2^k)$ and $F(x) := \lim_{k \to \infty} F_k(x)$.
\end{dfn}

Figure~\ref{fig:Fx} shows the function $F$ on $[0,1]$; the next theorem gives a result for $F$.

\begin{thm}
\label{thm:fx}
The function $F: [0,1] \to [0,1]$ is given by
\begin{align}
F(x) = \sum_{i=0}^{\infty} x_i 4^{\sum_{j=0}^{i-1} x_j} \left(\sum_{j=i}^{\infty} \frac{1-x_j}{2^j} \right)^2,
\end{align}
where we adopt a finite binary decimal number if $x \in [0, 1]$ has two binary expansions.
\end{thm}

\begin{proof}
When we set $x|_k = \sum_{i=0}^k x_i 2^{-i} \in [0,1]$, then
$t|_k = \sum_{i=0}^k t_i 2^i 
= 2^k \sum_{i=0}^k t_i 2^{i-k} 
= 2^k \sum_{i=0}^k t_{k-i} 2^{-i}
= 2^k \sum_{i=0}^k x_i 2^{-i}$.
Thus, for $x = t|_k/2^k \in [0,1]$,
\begin{align}
F_k(x) &:= \frac{num_a(t|_k)}{num_a(2^k)}\\
&= \frac{1}{4^k} \sum_{i=0}^k t_{k-i} 4^{\sum_{j=0}^i t_{k-j+1}} \left( 2^{k-i+1} - \sum_{j=0}^k t_j 2^j + \sum_{j=0}^i t_{k-j+1} 2^{k-j+1} \right)^2\\
&= \sum_{i=0}^k x_i 4^{\sum_{j=0}^{i-1} x_j} \left( 2^{-i+1} - \sum_{j=0}^k x_j 2^{-j} + \sum_{j=0}^{i-1} x_j 2^{-j} \right)^2\\
&= \sum_{i=0}^k x_i 4^{\sum_{j=0}^{i-1} x_j} \left( \sum_{j=i}^{\infty} 2^{-j} - \sum_{j=i}^k x_j 2^{-j} \right)^2.
\end{align}
Therefore,
\begin{align}
F(x) := \lim_{k \to \infty} F_k (x) 
= \sum_{i=0}^{\infty} x_i 4^{\sum_{j=0}^{i-1} x_j} \left(\sum_{j=i}^{\infty} \frac{1-x_j}{2^j} \right)^2.
\end{align}
The binary expansion of $x$ is unique, except the dyadic rationals 
$x= m/2^i$, which have two possible expansions. 
For $\tilde{x}= \sum_{i=0}^{l-1} x_i / 2^i + \sum_{i=l+1}^{\infty} x_i / 2^i$ and 
$\hat{x} = \sum_{i=0}^{l-1} x_i / 2^i + 1/2^l$, we have $\tilde{x} = \hat{x}$, but $F(\tilde{x}) \neq F(\hat{x})$. Thus, we adopt $F(\hat{x})$ for a finite binary decimal $\hat{x}$.
\end{proof}

We set $F(x) = \sum_{i=0}^{\infty} f_i(x)$, where
\begin{align}
\label{eq:fi}
f_i(x) &:= x_i 4^{\sum_{j=0}^{i-1} x_j} \left(\sum_{j=i}^{\infty} \frac{1-x_j}{2^j} \right)^2.
\end{align}
This function states that for $i=0$,
\begin{align}
f_0 (x) &= \left\{
\begin{array}{ll}
0 & \left( 0 \leq x < 1\right), \\
1 & \left( x=1 \right)
\end{array}
\right .
\end{align}
and if $i \in {\mathbb Z}_{>0}$,
\begin{align}
f_i(x) &= \left\{
\begin{array}{ll}
0 & \left( 2k/2^i \leq x < (2k+1)/2^i, \ x=1\right),\\
4^{\sum_{j=0}^{i-1} x_j} \left( \sum_{j=i+1}^{\infty} (1-x_j)/2^j \right)^2 & \left( (2k+1)/2^i \leq x < (2k+2)/2^i \right)
\end{array}
\right .
\end{align}
for $k = 0, 1, \ldots, (2^i-1)/2$.
Each function $f_i$ is piecewise continuous on $[2(2k+1)/2^i, (2k+2)/2^i)$.
Figure~\ref{fig:Fix} shows the graphs of $f_1$, $f_2$, $f_3$, and $f_4$.


\begin{figure}[htbp]
\begin{minipage}{0.5\hsize}
\begin{center}
\includegraphics[width=60mm]{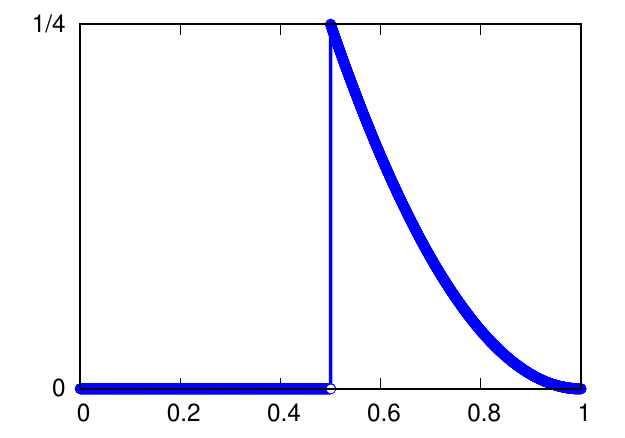}\\
(a) $f_1(x)$
\end{center}
\end{minipage}
\begin{minipage}{0.5\hsize}
\begin{center}
\includegraphics[width=60mm]{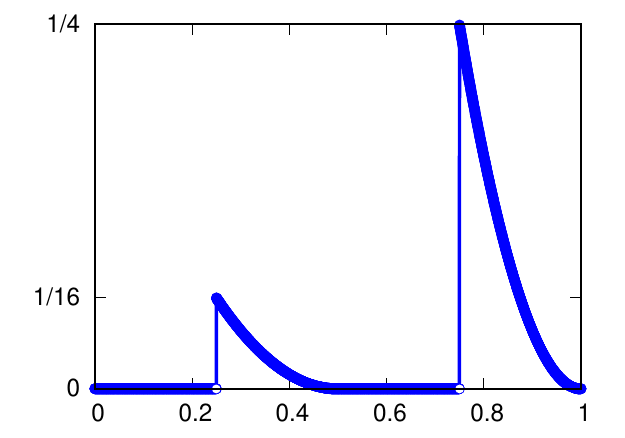}\\
(b) $f_2(x)$
\end{center}
\end{minipage}\\
\begin{minipage}{0.5\hsize}
\begin{center}
\includegraphics[width=60mm]{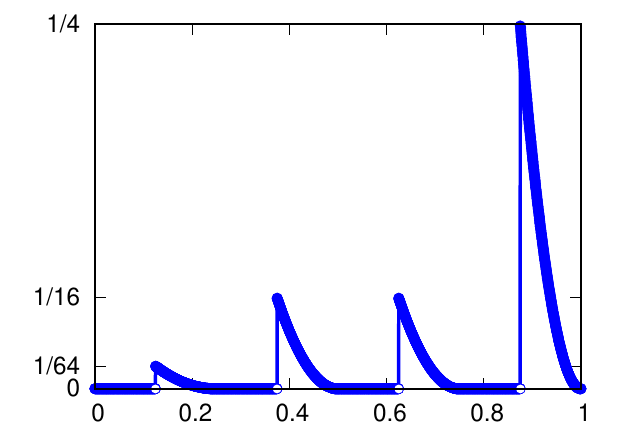}\\
(c) $f_3(x)$
\end{center}
\end{minipage}
\begin{minipage}{0.5\hsize}
\begin{center}
\includegraphics[width=60mm]{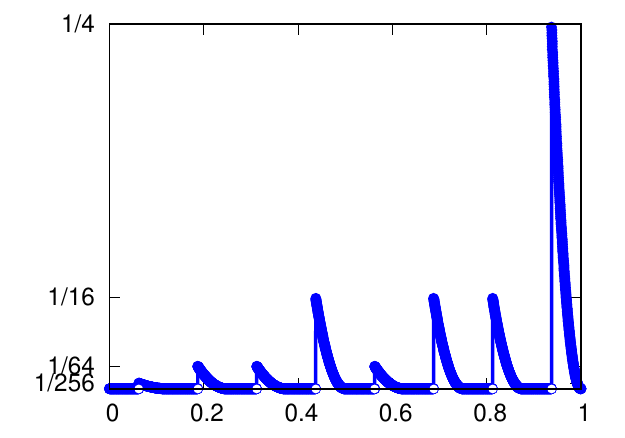}\\
(d) $f_4(x)$
\end{center}
\end{minipage}
\caption{$f_i(x)$ for $i=1, 2, 3$, and $4$}
\label{fig:Fix}
\end{figure}

Here, we describe the properties of the function $F$.
\begin{thm}
\label{thm:prop}
\begin{itemize}
\item[$(i)$] The function $F$ is Riemann integrable.
\item[$(ii)$] $\int_0^1 F(x) dx = 1/9$.
\end{itemize}
\end{thm}

Before we prove Theorem~\ref{thm:prop}, we show the following lemma.

\begin{lem}
\label{lem:bnd}
The function $F$ is on $[0, 1]$.
\end{lem}

\begin{proof}
It can be easily proven that $F(0)=0$ and $F(1)=1$ by the definition of $F$.

Now, we show that $F(x) < 1$ for $x \in (0, 1)$.
For $x = \sum_{i=0}^{\infty} x_i/2^i \in (0, 1)$, the sequence $(x_0 x_1 x_2 x_3 \cdots)$ is represented by $(O_1 I_1 O_2 I_2 O_3 I_3 \cdots)$,
where $O_m$ and $I_m$ are the $m$-th clusters of continuous $0$s and $1$s, respectively, in a binary number of $x$ for $m \in {\mathbb Z}_{>0}$. 
We denote $O_{m, k} = i$, where $x_i$ is the $k$-th element of $O_m$ for $1 \leq k \leq |O_m|$. 
We also denote $I_{m, k} = i$, where $x_i$ is the $k$-th element of $I_m$ for $1 \leq k \leq |I_m|$.
For $(x_0 x_1 x_2 x_3 x_4) = (0 1 0 0 1)$, for example, 
we have $O_1 = x_0$, $I_1 = x_1$, $O_2 = (x_2 x_3)$, and $I_3 = x_4$. 
Thus, we have $O_{1,1} = 0$, $I_{1,1} = 1$, $O_{2,1} = 2$, $O_{2,2} = 3$, and $I_{3,1} = 4$.

We prove the following three cases.

\begin{itemize}
\item[$(a)$] First, we consider infinite binary decimals $x$, except $x = \sum_{i=0}^{m-1} x_i/2^i + \sum_{i=m}^{\infty} 1/2^i$ for $m \in {\mathbb Z}_{>0}$.
In this case, for any $x_i=0$, there exists $l \in {\mathbb Z}_{>0}$ such that $x_{i+l} = 0$.
Then, we have $(O_1 I_1 O_2 I_2 O_3 I_3 \cdots)$ for $0 < |O_n|, |I_n| < \infty$ ($n \in {\mathbb Z}_{>0}$). 
Based on Equation~\eqref{eq:fi} for $i=I_{n,k}$,
\begin{align}
f_{I_{n,k}}(x) &= 4^{\sum_{j=1}^{n-1} |I_j|+k-1} \left( \sum_{i=n+1}^{\infty} \sum_{j=0}^{|O_i|-1} \frac{1}{2^{O_{i,1}+j}} \right)^2.
\end{align}
Thus, we have
\begin{align}
F(x) &= \sum_{i=1}^{\infty} f_i (x)
= \sum_{n=1}^{\infty} \sum_{k=1}^{|I_n|} f_{I_{n, k}}(x) \\
&= \sum_{n=1}^{\infty} 4^{\sum_{j=1}^{n-1} |I_j|} \left( \sum_{i=n+1}^{\infty} \sum_{j=0}^{|O_i|-1} \frac{1}{2^{O_{i,1}+j}} \right)^2 \sum_{k=1}^{|I_n|} 4^{k-1} \\
&= \frac{4}{3} \sum_{n=1}^{\infty} \left( 4^{\sum_{j=1}^n |I_j|} - 4^{\sum_{j=1}^{n-1} |I_j|} \right) \left( \sum_{i=n+1}^{\infty} \frac{1}{2^{O_{i,1}}} \left( 1 - \frac{1}{2^{|O_i|}} \right) \right)^2 \\
&< \frac{4}{3} \sum_{n=1}^{\infty} \left( 4^{\sum_{j=1}^n |I_j|} - 4^{\sum_{j=1}^{n-1} |I_j|} \right) \left( \sum_{i=n+1}^{\infty} \frac{1}{2^{O_{i,1}}} \right)^2 \\
&\leq \frac{4}{3} \sum_{n=1}^{\infty} \left( 4^{\sum_{j=1}^n |I_j|} - 4^{\sum_{j=1}^{n-1} |I_j|} \right) \frac{4}{4^{O_{n+1,1}}} \\
&< \frac{4^2}{3} \sum_{n=1}^{\infty} \frac{1}{4^{\sum_{j=1}^n |O_j|+1}} \\
&< \frac{4}{3} \sum_{n=1}^{\infty} \frac{1}{4} = \frac{4}{3} \cdot \frac{1}{3} = \frac{4}{9}.
\end{align}
\end{itemize}
The other cases can be proved in a similar manner.
\begin{itemize}
\item[$(b)$] For a finite binary number $x = \sum_{i=0}^l x_i/2^i$, there exists a positive integer $m (\leq l)$ such that $(O_1 I_1 O_2 I_2 \cdots O_{m-1} I_{m-1} O_m)$ for $0 < |O_n|, |I_n| < \infty$ ($1 \leq n < m$) and $|O_m| = \infty$.
\begin{align}
F(x) &= \sum_{n=1}^{m-1} \sum_{k=1}^{|I_n|} f_{I_{n, k}}(x) \\
&= \frac{4}{3} \sum_{n=1}^{m-1} \left( 4^{\sum_{j=1}^n |I_j|} - 4^{\sum_{j=1}^{n-1} |I_j|} \right) \left( \sum_{i=n+1}^m \frac{1}{2^{O_{i,1}}} \left( 1 - \frac{1}{2^{|O_i|}} \right) \right)^2 \\
&< \frac{4}{3} \sum_{n=1}^{m-1} \frac{1}{4} = \frac{4}{9}.
\end{align}
\item[$(c)$] For a finite binary number $x = \sum_{i=0}^{l-1} x_i/2^i + \sum_{i=l}^{\infty} 1/2^i$, there exists a positive integer $m (\leq l)$ such that $(O_1 I_1 O_2 I_2 \cdots O_m I_m)$ for $0 < |O_n|, |I_n| < \infty$ ($1 \leq n < m$), $1 \leq |O_m| < m$ and $|I_m| = \infty$.
\begin{align}
F(x) &= \sum_{n=1}^{m-1} \sum_{k=1}^{|I_n|} f_{I_{n, k}}(x) 
+ \sum_{k=1}^{\infty} f_{I_{m, k}}(x) 
= \sum_{n=1}^{m-1} \sum_{k=1}^{|I_n|} f_{I_{n, k}}(x) 
< \frac{4}{9}.
\end{align}
\end{itemize}
Therefore, $F(x) \in [0,1]$ for $x \in [0,1]$.
\end{proof}

\begin{proof}[Proof of Theorem~\ref{thm:prop}~$(i)$]
By Lemma~\ref{lem:bnd}, $F$ is bounded.
Let $U$ denote the set of points of $(0, 1)$ that have an infinite binary expansion.
According to the constructing method of $f_i$, $F$ is continuous at $0$ and points in $U$
(if the binary expansion of $x$ is finite except $0$, $F$ is discontinuous).
The function $F$ is Riemann integrable because it has only countable discontinuous points.
\end{proof}

\begin{figure}[H]
\begin{center}
\includegraphics[width=.6\linewidth]{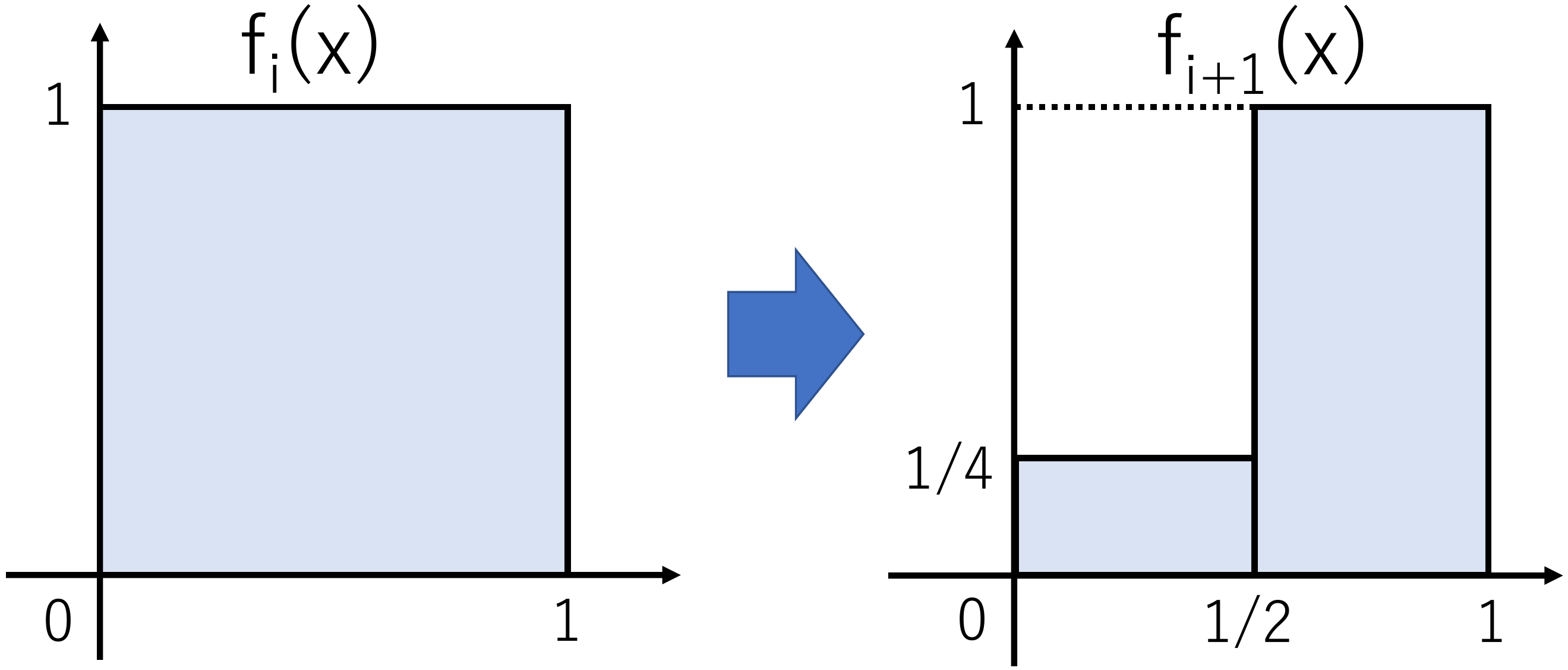}
\end{center}
\caption{Self-affine set $f_{i+1}(x)$, union of affine copies of $f_i(x)$}
\label{fig:fsim}
\end{figure}

\begin{proof}[Proof of Theorem~\ref{thm:prop}~$(ii)$]
By Theorem~\ref{thm:prop}~$(i)$, $F$ is Riemann integrable.
The sum of the sequence $\sum_{i=0}^n f_i(x)$ does not uniformly converge to $F$. However, it is uniformly finite and converges point-wise. 
Thus, $F$ is term-wise integrable by Arzela's theorem.

According to the definition of $f_i$, it can be easily demonstrated that $f_1(x)$ on $1/2 \leq x < 1$ is equivalent to a quadratic function $(x-1)^2$.
Figure~\ref{fig:fsim} shows the self-affine set $f_{i+1}(x)$ derived from $f_i(x)$ for any $i \in {\mathbb Z}_{> 0}$.
The right part of $f_{i+1}(x)$ on $1/2 \leq x \leq 1$ 
is generated by reducing the breadth of $f_i(x)$ by half. 
The left part of $f_{i+1}(x)$ on $0 \leq x \leq 1/2$ 
is generated by reducing the breadth and length of $f_i(x)$ by half and a quarter, respectively.
Owing to the self-affinity of $f_i(x)$, we have $\int_0^1 f_{i+1}(x) dx = (1/2 + 1/8) \int_0^1 f_i(x) dx$ for $i \in {\mathbb Z}_{>0}$. 
Therefore, 
\begin{align}
\int_0^1 F(x) dx &= \sum_{i=0}^{\infty} \int_0^1 f_i(x) dx 
= \int_0^1 f_1(x) dx \sum_{i=0}^{\infty} \left( \frac{5}{8} \right)^i
= \frac{1}{24} \cdot \frac{8}{3} = \frac{1}{9}.
\end{align}
\end{proof}


\subsection{Numbers of nonzero states for $2$dECA $T_b$ and given function}

In this section, we exmine the orbit of $2$dECA $T_b$.
Figure~\ref{fig:st696} shows the spatio-temporal pattern $T_b^t x_o$ for $0 \leq t \leq 15$.
The spatial and spatio-temporal patterns for $T_b$ hold self-similarity as effectively as the patterns of $T_a$ and the structure that consists of quadrangular pyramids is similar to that of $T_a$.
The spatial pattern $T_b^t x_o$ for each time step $t$ consists of four congruent patterns, which are congruent with the spatial pattern of time step $t-2^k$, and one square with the length of $2^{k+1}-t$ cells.
In the spatio-temporal pattern $\{ T_b^t x_o\}_{t=0}^{2^k}$, the squares consist of a quadrangular pyramid whose apex is located at the bottom.


\begin{figure}[H]
\begin{center}
\includegraphics[width=1.\linewidth]{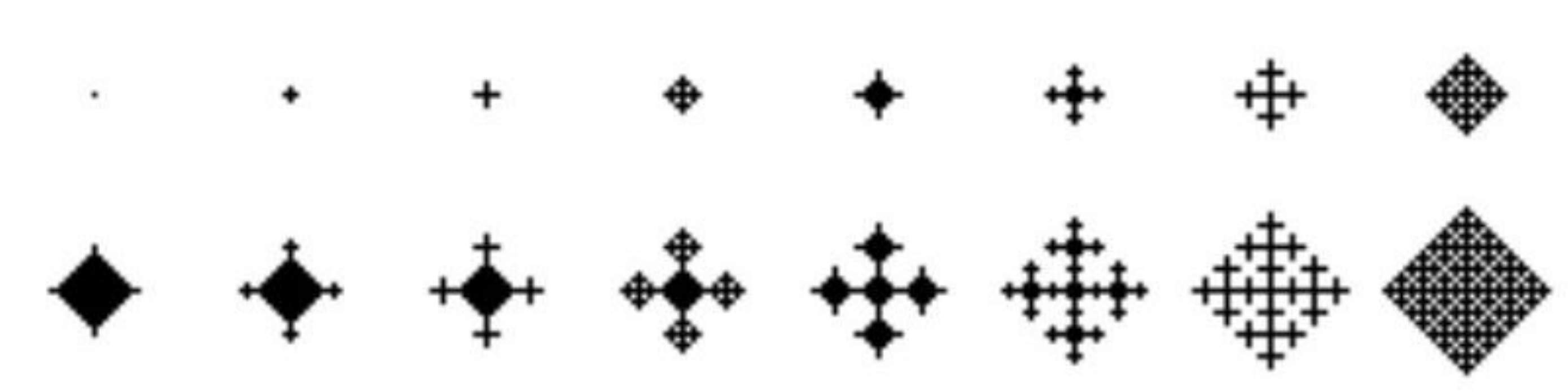}
\end{center}
\caption{Spatio-temporal pattern of $T_b^t x_o$ for $0 \leq t \leq 15$}
\label{fig:st696}
\end{figure}

We obtain the following results regarding the values of $num_b(t)$ and $cum_b(2^k)$.

\begin{prop}
\label{prop:num696}
For $t = \sum_{i=0}^k t_i 2^i \geq 0$ with $k \in {\mathbb Z}_{\geq 0}$ such that $t_k=1$ and $t_i=0$ for $i > k$, we have
\begin{align}
num_b(t) &= \sum_{i=0}^k t_{k-i} 4^{\sum_{j=0}^i t_{k-j+1}} A\left(\sum_{j=0}^{k-i} t_j 2^j \right),\\
cum_b(2^k) &= \frac{8 \cdot 5^k}{9} + \frac{2}{9} 8^k - \frac{1}{9} 2^k,
\end{align}
where $A(t) = 2 (2^{k+1}-t+1)(2^{k+1}-t) +1 - (2/3) (2^k+1)(2^k+2) \prod_{j=0}^{k-1} (1-t_j)$.
\end{prop}

\begin{proof}
First, we calculate the number of nonzero states in the central square of a spatial pattern $T_b^t x_o$, which is a slice of $Q_k$ (see Figure~\ref{fig:sph} (a)).
The length of the square for time step $t$ is $2^{k+1}-t$ and the number of nonzero states in the square is
$2 \left( \sum_{i=1}^{2^{k+1}-t-1} (2i-1) \right) + 2(2^{k+1}-t)-1 
= 2 (2^{k+1}-t)(2^{k+1}-t-1) +1$.
Considering a correction amount for $t=2^k$, we define
\begin{align}
A(t) 
&= 2 (2^{k+1}-t+1)(2^{k+1}-t) +1 - \frac{2}{3} (2^k+1)(2^k+2) \prod_{j=0}^{k-1} (1-t_j)
\end{align}
for $2^k \leq t < 2^{k+1}$.
According to the self-similarity of the spatio-temporal pattern, we have the following equation. 
\begin{align}
\label{eq:4tk2}
\begin{cases}
num_b(0) &= 0, \\
num_b(1) &= 1,\\
num_b(t) &= 4^{t_k} \, num_b(t-t_k 2^k) + t_k A(t).
\end{cases}
\end{align}
By Equation~\eqref{eq:4tk2}, we have
$num_b(t-t_k 2^k) = 4^{t_{k-1}} num_b(t- t_k 2^k - t_{k-1} 2^{k-1}) + t_{k-1} A(t-t_k 2^k)$. Thus,
\begin{align}
num_b(t) 
&= 4^{t_k} \left( 4^{t_{k-1}} num_b(t- t_k 2^k - t_{k-1} 2^{k-1}) + t_{k-1} A(t-t_k 2^k) \right) + t_k A(t)\\
&= 4^{t_k + t_{k-1}} \, num_b(t- t_k 2^k - t_{k-1} 2^{k-1})  
+ 4^{t_k} t_{k-1} A(t-t_k 2^k) + t_k A(t).
\end{align}
Thorough inductive calculation, we have 
\begin{align}
num_b(t) &= 4^{\sum_{i=0}^k t_{k-i}} \, num_b \left(t- \sum_{i=0}^k t_{k-i} 2^{k-i} \right) \nonumber \\
& \qquad + \sum_{i=0}^k 4^{\sum_{j=0}^i t_{k-j+1}} t_{k-i} A\left(t- \sum_{j=0}^i t_{k-j+1} 2^{k-j+1} \right).
\end{align}
Because $num_b \left(t- \sum_{i=0}^k t_{k-i} 2^{k-i} \right) = num_b(0) = 0$, 
\begin{align}
\label{eq:num01}
num_b(t) &= \sum_{i=0}^k t_{k-i} 4^{\sum_{j=0}^i t_{k-j+1}} A\left(t- \sum_{j=0}^i t_{k-j+1} 2^{k-j+1} \right)\\
&= \sum_{i=0}^k t_{k-i} 4^{\sum_{j=0}^i t_{k-j+1}} A\left(\sum_{j=0}^k t_j 2^j - \sum_{j=0}^i t_{k-j+1} 2^{k-j+1} \right)\\
&= \sum_{i=0}^k t_{k-i} 4^{\sum_{j=0}^i t_{k-j+1}} A\left(\sum_{j=0}^{k-i} t_j 2^j \right).
\end{align}

Next, we calculate $cum_b(2^k)$.
Let $y_k = cum_b(2^k)$ and let $z_k$ be the number of nonzero states in the quadrangular pyramid $Q_k$ until time step $2^k-1$.
Thus, we have
\begin{align}
\begin{cases}
y_0 = 1,\\
z_{k+1} = \sum_{i=1}^{2^k} (2i^2-2i+1)
= (2 (2^k)^3 + 2^k)/3 \\
y_{k+1} = 5 y_k + z_{k+1}.
\end{cases}
\end{align}
Consequently, 
\begin{align}
y_k &= \frac{8 \cdot 5^k}{9} + \frac{2}{9} 8^k - \frac{1}{9} 2^k = cum_b(2^k).
\end{align}
\end{proof}

\begin{figure}[htbp]
\begin{minipage}{0.5\hsize}
\begin{center}
\includegraphics[width=70mm]{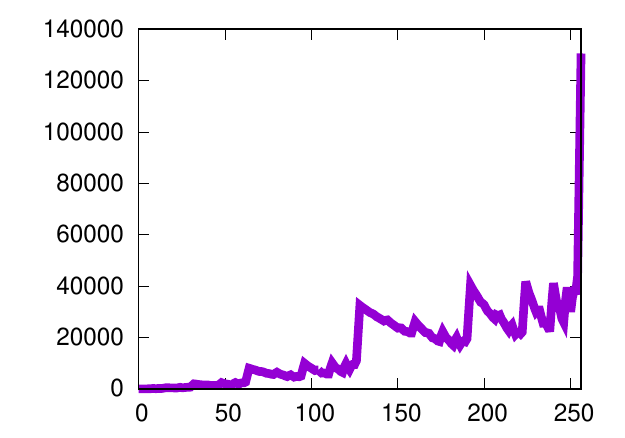}
\end{center}
\caption{$num_b(t)$ for $1 \leq t \leq 256$}
\label{fig:num696}
\end{minipage}
\begin{minipage}{0.5\hsize}
\begin{center}
\includegraphics[width=70mm]{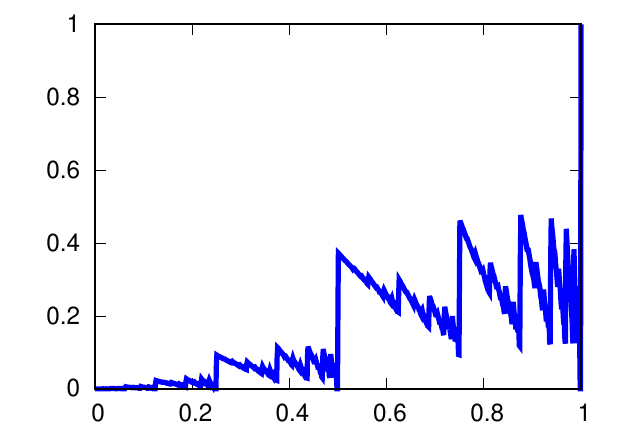}
\end{center}
\caption{$G(x)$}
\label{fig:Gx}
\end{minipage}
\end{figure}

\begin{dfn}
For $t|_k= \sum_{i=0}^k t_i 2^i (\leq 2^k)$, let $x = t|_k/2^k \in [0,1]$.
We define $G_k(x) := num_b(t|_k)/num_b(2^k)$ and $G(x) := \lim_{k \to \infty} G_k(x)$.
\end{dfn}

For the function $G$, we obtain the following result. 

\begin{thm}
\label{thm:gx}
The function $G: [0, 1] \to [0, 1]$ is given by
\begin{align}
G(x) = \sum_{i=0}^{\infty} x_i 4^{\sum_{j=0}^{i-1} x_j} \left( \frac{3}{2} \left( \sum_{j=i}^{\infty} \frac{1-x_j}{2^j} \right)^2 - \frac{1}{2 \cdot 4^i} \prod_{j=i+1}^{\infty} (1-x_j) \right),
\end{align}
where we adopt a finite binary decimal number if $x \in [0, 1]$ has two binary expansions.
\end{thm}

\begin{proof}
For $t|_k= \sum_{i=0}^k t_i 2^i (\leq 2^k)$, we have
$t|_k = \sum_{i=0}^k t_i 2^i 
= 2^k \sum_{i=0}^k t_i 2^{i-k} 
= 2^k \sum_{i=0}^k t_{k-i} 2^{-i}
= 2^k \sum_{i=0}^k x_i 2^{-i}$.
For $x = t|_k/2^k =\sum_{i=0}^k x_i 2^{-i} \in [0,1]$, we have
\begin{align}
G_k(x) &= \frac{num(t|_k)}{num(2^k)}
= \sum_{i=0}^k t_{k-i} 4^{\sum_{j=0}^i t_{k-j+1}} \frac{A\left(\sum_{j=0}^{k-i} t_j 2^j \right)}{A(2^k)},
\end{align}
where
\begin{align}
A\left(\sum_{j=0}^{k-i} t_j 2^j \right)
&= 8 \cdot 4^{k-i} - 4 \left(2 \left(\sum_{j=0}^{k-i} t_j 2^j \right)-1 \right) 2^{k-i} + 2 \left(\sum_{j=0}^{k-i} t_j 2^j \right)^2 \nonumber \\
&\qquad -2 \left(\sum_{j=0}^{k-i} t_j 2^j \right) + 1 - \frac{2}{3} (2^{k-i}+1)(2^{k-i}+2) \prod_{j=0}^{k-i-1} (1-t_j)\\
&= 4^k \left( \frac{8}{4^i} - \frac{8}{2^i} \sum_{j=0}^{k-i} t_j \frac{2^j}{2^k} + \frac{4}{2^{k+i}} + 2 \left(\sum_{j=0}^{k-i} t_j \frac{2^j}{2^k} \right)^2 - \frac{2}{2^k} \sum_{j=0}^{k-i} t_j \frac{2^j}{2^k} \right . \nonumber \\
&\qquad \qquad \left . + \frac{1}{4^k} - \frac{2}{3} \left(\frac{1}{2^i}+\frac{1}{2^k} \right) \left(\frac{1}{2^i}+\frac{2}{2^k} \right) \prod_{j=0}^{k-i-1} (1-t_j) \right).
\end{align}
Because $\sum_{j=0}^{k-i} t_j 2^j/2^k 
= \sum_{j=0}^{k-i} x_{k-j} 2^{-(k-j)}
= \sum_{j=i}^k x_j 2^{-j}
= \sum_{j=0}^k x_j 2^{-j} - \sum_{j=0}^{i-1} x_j 2^{-j}$,
\begin{align}
\frac{A\left(\sum_{j=0}^{k-i} t_j 2^j \right)}{4^k}
&= \frac{8}{4^i} - \frac{8}{2^i} \left(\sum_{j=0}^k x_j 2^{-j} - \sum_{j=0}^{i-1} x_j 2^{-j} \right) + \frac{4}{2^{k+i}} \nonumber\\
& \qquad + 2 \left(\sum_{j=0}^k x_j 2^{-j} - \sum_{j=0}^{i-1} x_j 2^{-j} \right)^2 - \frac{2}{2^k} \left(\sum_{j=0}^k x_j 2^{-j} - \sum_{j=0}^{i-1} x_j 2^{-j} \right) \nonumber\\
& \qquad + \frac{1}{4^k} - \frac{2}{3} \left(\frac{1}{2^i}+\frac{1}{2^k} \right) \left(\frac{1}{2^i}+\frac{2}{2^k} \right) \prod_{j=i+1}^k (1-x_j).
\end{align}
Hence, 
\begin{align}
G(x) &= \lim_{k \to \infty} G_k (x)\\
&= \lim_{k \to \infty} \sum_{i=0}^k t_{k-i} 4^{\sum_{j=0}^i t_{k-j+1}} \frac{A\left(\sum_{j=0}^{k-i} t_j 2^j \right)}{A(2^k)}\\
&= \sum_{i=0}^{\infty} x_i 4^{\sum_{j=0}^{i-1} x_j} \left( \frac{3}{2} \left( \sum_{j=i}^{\infty} \frac{1-x_j}{2^j} \right)^2 - \frac{1}{2 \cdot 4^i} \prod_{j=i+1}^{\infty} (1-x_j) \right).
\end{align}
\end{proof}

Next, we describe the properties of $G$.

\begin{thm}
\label{thm:prop2}
\begin{itemize}
\item[$(i)$] The function $G$ on $[0, 1]$ is Riemann integrable.
\item[$(ii)$] $\int_0^1 G(x) dx = 1/6$.
\end{itemize}
\end{thm}

Next, we prepare the following lemma.

\begin{lem}
\label{lem:GFH}
We show the relationship between $F$ and $G$ using the function $H_k$ in Remark~\ref{rmk:hx}.
\begin{align}
G(x) &= \sum_{i=0}^{\infty} x_i 4^{\sum_{j=0}^{i-1} x_j -1} \left( \frac{3}{2} \left( \sum_{j=i}^{\infty} \frac{1-x_j}{2^{j-i}} \right)^2 - \frac{1}{2} \prod_{j=i+1}^{\infty} (1-x_j) \right)\\
&= \frac{3}{2} \sum_{i=0}^{\infty} x_i 4^{\sum_{j=0}^{i-1} x_j} \left( \sum_{j=i}^{\infty} \frac{1-x_j}{2^j} \right)^2 - \frac{1}{2} \sum_{i=0}^{\infty} x_i 4^{\sum_{j=0}^{i-1} x_j -1} \prod_{j=i+1}^{\infty} (1-x_j)\\
&= \frac{3}{2} F(x) - \frac{1}{2} \sum_{i=0}^{\infty} x_i 4^{\sum_{j=0}^{i-1} x_j -1} \prod_{j=i+1}^{\infty} (1-x_j)\\
&= \frac{3}{2} F(x) - \frac{1}{2} H_{k_x} \left(x + \frac{1}{2^{k_x}} \right),
\end{align}
where $k_x \in {\mathbb Z}_{\geq 0}$ satisfies $x_{k_x}=1$ and $x_i = 0$ for all $i > k_x$.
\end{lem}

The lemma revealed the similarities and differences between $F$ and $G$.

\begin{proof}[Proof of Theorem~\ref{thm:prop2}~$(i)$]
As $H_{k_x}(0)= 0$ and $H_{k_x}(1)= 1$, $G(0)=0$ and $G(1) = 3/2 \cdot 1 - 1/2 \cdot 1 = 1$.
By lemmas~\ref{lem:bnd} and \ref{lem:GFH}, for $x \in (0,1)$, $G(x) < 3/2 \cdot 4/9 = 2/3$. 
Hence, $G(x) \in [0,1]$.

The functions $F$ and $H_{k_x}$ are Riemann integrable by Theorem~\ref{thm:prop}~$(i)$ and Remark~\ref{rmk:hx}, respectively.
Hence, $G$ is Riemann integrable.
\end{proof}

\begin{proof}[Proof of Theorem~\ref{thm:prop2}~$(ii)$]
By Theorem~\ref{thm:prop}~$(ii)$ and Lemma~\ref{lem:GFH}, 
\begin{align}
\int_0^1 G(x) dx =& \frac{3}{2} \cdot \frac{1}{9} + \frac{1}{2} \cdot 0 = \frac{1}{6}.  
\end{align}
\end{proof}

\section{Conclusions and future works}
\label{sec:cr}

We wrote down the number of nonzero states of the spatial and spatio-temporal patterns of two two-dimensional elementary cellular automata.
We presented two discontinuous Riemann integrable functions, $F$ and $G$, on the unit interval based on the dynamics of the automata.
They have similar properties because the structures of the limit sets are similar and hold self-similarities.
We also showed that $G$ is represented by $F$ and $H_k$, which is given by the orbit of another two-dimensional elementary cellular automaton $T_{S0}$, and $cum_{S0}(t)$ is represented by Salem's singular function $L_{1/5}$.

In future work, we plan to determine the other pathological functions based on the dynamics of elementary cellular automata.
In this study and previous works, we demonstrated the existence of singular functions and discontinuous Riemann integrable functions.
We will study the orbits of cellular automata that hold self-similarity and provide new pathological functions.

\section*{Acknowledgment}
This work is partly supported by a Grant-in-Aid for Scientific Research (18K13457) funded by the Japan Society for the Promotion of Science.


\end{document}